\documentclass{amsart}
\usepackage{amssymb,latexsym}
\usepackage{amsmath}
\usepackage{graphicx}
\usepackage{amscd}
\usepackage{color}
\usepackage{enumerate}
\newenvironment{enumeratei}{\begin{enumerate}[\upshape (i)]}{\end{enumerate}}
\numberwithin{equation}{section}
\theoremstyle{plain}
 \newtheorem{theorem}{Theorem}[section]
 \newtheorem{lemma}[theorem]{Lemma}
 \newtheorem{proposition}[theorem]{Proposition}
 
 \newtheorem{corollary}[theorem]{Corollary}
 
\theoremstyle{definition}
 \newtheorem{definition}[theorem]{Definition}
 \newtheorem{remark}[theorem]{Remark}

\newcommand \datum {June 29, 2017}

\newcommand \pcrely {principal congruence representability}
\newcommand \szn {\textup{col}}

\renewcommand\rho{\varrho}
\newcommand\length{\textup{length}}
\newcommand\inp{\mathfrak p}
\newcommand\inh{\mathfrak h}
\newcommand\inq{\mathfrak q}
\newcommand\inr{\mathfrak r}
\newcommand\ins{\mathfrak s}
\newcommand\kulse[2] {e(#1,#2)}
\newcommand\belse[2] {i(#1,#2)}
\newcommand\kulsh[2] {E(#1,#2)}
\newcommand\belsh[2] {I(#1,#2)}
\newcommand \symi [2] {[#1,#2]^\ast}
\newcommand \Con  {\textup{Con}}
\newcommand \Prime  {\textup{Prime}}
\newcommand \erep  {\textup{erep}}
\newcommand \srep  {\textup{SRep}}
\newcommand \hei {\mathfrak h}

\newcommand\ideal[1]{\mathord\downarrow #1}
\newcommand\filter[1]{\mathord\uparrow #1}
\newcommand \tuple [1] {\langle #1\rangle}
\newcommand \pair [2] {\tuple{#1,#2}}
\newcommand \tripl [3] {\tuple{#1,#2,#3}}
\newcommand \tbf [1] {\textbf{#1}} 
\newcommand \set[1] {\{#1\}}
\newcommand \jj {\vee}
\newcommand \Princ {\textup{Princ}}
\newcommand \Intv {\textup{Intv}}
\newcommand \con {\textup{con}}

\newcommand \lbound {\textup{Bnd}_{\textup{left}}}
\newcommand \rbound {\textup{Bnd}_{\textup{right}}}
\newcommand \ljsp {\textup{ljs}}
\newcommand \rjsp {\textup{rjs}}

\newcommand \flrep{\textup{(f\kern0.7pt L\kern -0.7pt)}}
\newcommand \alrep{\textup{(L\kern -0.7pt)}}

\newcommand \aarep{\textup{(A)}}
\newcommand \aaerep{\textup{(A1)}}
\renewcommand\phi{\varphi}
\renewcommand\epsilon{{\boldsymbol\varepsilon}}
\newcommand\balpha{{\boldsymbol\alpha}}
\newcommand\bbeta{{\boldsymbol\beta}}
\newcommand\bgamma{{\boldsymbol\gamma}}
\newcommand\bdelta{{\boldsymbol\delta}}
\newcommand\bmu{{\boldsymbol\mu}}
\newcommand\alga {\mathfrak A}
\newcommand\semmi [1] {}

\begin{document}
\title
[Principal congruences in a distributive congruence lattice]
{On the set of principal congruences in a distributive congruence lattice of an algebra}

\author[G.\ Cz\'edli]{G\'abor Cz\'edli}
\email{czedli@math.u-szeged.hu}
\urladdr{http://www.math.u-szeged.hu/\textasciitilde{}czedli/}
\address{University of Szeged\\ Bolyai Institute\\Szeged,
Aradi v\'ertan\'uk tere 1\\ Hungary 6720}

\thanks{This research was supported by
NFSR of Hungary (OTKA), grant number K 115518}

\begin{abstract} 
Let $Q$ be a subset of a finite distributive lattice $D$. An algebra $A$ \emph{represents the inclusion $Q\subseteq D$ by principal congruences} if  the congruence lattice of $A$ is isomorphic to $D$ and  the ordered set of principal congruences of $A$ corresponds to $Q$ under this isomorphism. If there is such an algebra for \emph{every} subset $Q$ containing $0$, $1$, and all join-irreducible elements of $D$, then $D$ is said to be \emph{fully \aaerep-rep\-res\-entable}.
We prove that every fully \aaerep-representable finite distributive lattice  is planar and it has at most one join-reducible coatom. Conversely, we prove that every finite planar distributive lattice with at most one join-reducible coatom  is \emph{fully chain-representable} in the sense of a recent paper of G.\ Gr\"atzer. Combining the results of this paper with another paper by the present author, it follows that every fully \aaerep-representable  finite distributive lattice  is ``fully representable'' even by principal congruences of \emph{finite lattices}. Finally, we prove that every \emph{chain-representable} inclusion $Q\subseteq D$ can be represented by the principal congruences of a finite (and quite small) algebra.
\end{abstract}

\subjclass {06B10{{\hfill \color{red}
\datum{}\quad {\tiny(Novelty: Prop.~\ref{propositionbchrsnlsgR}, Sect.~\ref{sectionSnglPrf})}\color{black}}}}

\maketitle
\section{Introduction and results}
Gr\"atzer~ \cite[Probl.\ 12]{gG13} and \cite[Probl.\ 22.1]{ggCFL2} raised the problem of 
characterizing lattices and their subsets that can be represented simultaneously as  congruence lattices and the sets of principal congruences, respectively, of algebras or lattices. The first steps in this direction were made by Gr\"atzer~\cite{ggwith1} and Gr\"atzer and Lakser~\cite{gratzerlakser}; here we continue their investigations.
For a finite lattice $L$,  $J(L)$ denotes the ordered set of nonzero join-irreducible elements of $L$, $J_0(L)$ stands for $J(L)\cup\set0$, and we let $J^+(L)=J(L)\cup\set{0,1}$. For an algebra $A$, let $\Con(A)$ be the \emph{congruence lattice} of $A$, while $\Princ(A)$ will stand for the ordered set of \emph{principal congruences} of $A$. The algebra $A$ can be infinite but we always assume that $\Con(A)$ is finite. Then every congruence of $A$ is the join of finitely many principal congruences, whereby 
\begin{equation}
J_0(\Con(A)) \subseteq \Princ (A)\subseteq \Con(A).
\label{eqAtgjWrSD}
\end{equation}
For a subset $Q$ of a finite lattice $D$,  an algebra $A$ \emph{represents the inclusion}  $Q\subseteq D$ \emph{by principal congruences}  if there exists an isomorphism $\phi\colon \Con(A)\to D$ such that $Q=\phi(\Princ(A))$. In this case, we also say that the \emph{inclusion} $Q\subseteq D$ is \emph{represented by the principal congruences of} $A$. Mostly, we consider only the case where $D$ is distributive. Our first aim is to prove the following statement; condition \eqref{eqKzGtrW} in it is motivated by \eqref{eqAtgjWrSD}. Note that this section does not contain proofs; they are given in the subsequent sections. 

\begin{proposition}\label{propdhmhZ} Let $D$ be a finite distributive lattice. Then the following two conditions are equivalent.
\begin{enumerate}[\upshape(a)]
\item For every set $Q$ such that
\begin{equation} J_0(D)\subseteq Q\subseteq D,
\label{eqKzGtrW}
\end{equation}
the inclusion $Q\subseteq D$ is represented by the principal congruences of some algebra~$A$. (This condition will be called \emph{full \aarep-representability}.)
\item The lattice $D$  is a planar and  $1_D\in J_0(D)$, that is, $|D|=1$ or $D$ has exactly one coatom.
\end{enumerate}
\end{proposition}

In connection with this statement, see also Corollary~\ref{corolsrpBngN} later. 

A finite lattice  $D$ will be called \emph{fully \aarep-representable} if it satisfies 
condition (a) from Proposition~\ref{propdhmhZ}.  The notation \aarep{} comes from ``\emph{algebra}''. 

\begin{corollary}\label{crollarczhTngPr}
Let $V=\set{0,\balpha,\bbeta}$ be the ``V-shaped'' three-element ordered set with smallest element 0 and maximal elements $\balpha$ and $\bbeta$. Then $\Princ(A)\cong V$ holds for no algebra $A$.
\end{corollary}

It has previously been known that $V$ cannot be represented as $\Princ(L)$ of a \emph{lattice} $L$, since we know from Gr\"atzer~\cite{gG13}, see also (1.1) in Cz\'edli~\cite{czgaleph0}, that $\Princ(L)$ is always a directed ordered set but $V$ is not.
Corollary~\ref{crollarczhTngPr} indicates why it would be difficult to extend the results of 
Cz\'edli~\cite{czgonemap}, \cite{czgaleph0}, \cite{czgprincout}, \cite{czgmanymaps}, and \cite{czgcometic} and Gr\"atzer~\cite{gG13}, \cite{gGprincII}, and \cite{gGprincIII} from
the representability of ordered sets by  principal \emph{lattice} congruences to that by arbitrary principal congruences.

Using \eqref{eqAtgjWrSD} and  that $\Princ(L)$ is a directed ordered set, if $\Con(L)$ is finite, then $J_0(\Con(L))$ has an upper bound in  $\Princ(L)$. This upper bound  is necessarily the top element of $\Con(L)$. Hence, for every lattice $L$ such that $\Con(L)$ is finite,
\begin{equation}
J^+(\Con(L)) \subseteq \Princ (L)\subseteq \Con(L).
\label{eqLtgjWrSD}
\end{equation}
Our main goal is to prove the following theorem;  \eqref{eqTzhjP} is motivated by \eqref{eqLtgjWrSD}.

\begin{theorem}\label{thmmain} Let $D$ be a finite distributive lattice, and consider the following three conditions on $D$.
\begin{enumeratei}
\item\label{thmmaina} For every $Q$ satisfying
\begin{equation} J^+(D)\subseteq Q\subseteq D,
\label{eqTzhjP}
\end{equation}
the inclusion $Q\subseteq D$ is represented by the principal congruences of some \emph{algebra} $A$; if this condition holds then
$D$ is said to be \emph{fully $\aaerep$-representable}. 
\item\label{thmmainc}
For every $Q$ satisfying  \eqref{eqTzhjP},
the inclusion $Q\subseteq D$ is represented by the principal congruences of some \emph{finite lattice} $L$;  
if this condition holds then $D$ is said to be \emph{fully $\flrep$-representable}.
\item\label{thmmaind}
 $D$  is planar and it has at most one join-reducible coatom.
\end{enumeratei}
Then \eqref{thmmaina} implies \eqref{thmmaind} and  the trivial implication \eqref{thmmainc} $\Rightarrow$ \eqref{thmmaina}  also holds.
\end{theorem}

The notation \aaerep{} in Theorem~\ref{thmmain} comes from \emph{algebra} and $1\in J^+(D)\subseteq Q$, while  \flrep{} comes from  \emph{finite lattice}. The concept of full \flrep{}-representability and, more generally, the representability of just one subset $Q$ of $D$ by principal congruences of a finite lattice $L$ are taken from Gr\"atzer~\cite{ggwith1} and Gr\"atzer and Lakser~\cite{gratzerlakser}. 

\begin{remark}\label{remarktLpR} 
Cz\'edli~\cite{czgDQprincrep}  gave\footnote{After \texttt{https://arxiv.org/abs/1705.10833v1}, the first version of the present paper.
While  \cite{czgDQprincrep} is mainly for some specialists of lattice theory, the present paper is written for a wider readership.} a long proof for the implication  \eqref{thmmaind} $\Rightarrow$ \eqref{thmmainc}. Hence, for a finite distributive lattice $D$,   \eqref{thmmaina},  \eqref{thmmainc},  and \eqref{thmmaind} in Theorem~\ref{thmmain} are equivalent conditions. In particular,  \eqref{thmmaina} $\Rightarrow$ \eqref{thmmainc}, which seems to be a surprise. 
\end{remark}

Next, we need the following concept, introduced in Gr\"atzer~\cite{ggwith1}.

\begin{definition}\label{defsznzSj} 
Let $D$ be a finite distributive lattice. 
\begin{enumeratei}
\item 
By a \emph{$J(D)$-colored chain} we mean a triplet $\tripl C\szn D$ such that $C$ is a finite chain, $D$ is a finite distributive lattice, and $\szn\colon \Prime(C)\to J(D)$ is a surjective map from the \emph{set $\Prime(C)$ of all prime intervals} of $C$ onto $J(D)$. If $\inp\in\Prime(C)$, then $\szn(\inp)$ is the \emph{color} of the edge $\inp$. 
\item Given a $J(D)$-colored chain $\tripl C\szn D$, we define a map denoted by $\erep$ from the \emph{set $\Intv(C)$ of all intervals} of  $C$ onto $D$ as follows: 
for $I\in\Intv(C)$,  let 
\begin{equation}
\erep(I):=\bigvee_{\inp\in\Prime(I)} \szn(\inp)\,;
\label{eqsdhsdrepRtp} 
\end{equation}
the join is taken in $D$ and $\erep(I)$ is called the \emph{element represented} by $I$.
\item The set 
\[\srep(C,\szn,D):=\set{\erep(I): I\in \Intv(C)}\] 
will be called the  \emph{set represented} by the $J(D)$-colored chain $\tripl C\szn D$. Clearly,  by the surjectivity of $\szn$,  $Q:=\srep(C,\szn,D)$ satisfies \eqref{eqTzhjP} in this case.  
\item For a subset $Q$ of $D$, the inclusion $Q\subseteq D$ is \emph{chain-representable} if there exists a $J(D)$-colored chain $\tripl C\szn D$ such that $Q=\srep(C,\szn, D)$. Note that $C$ need not be a subchain of $D$.
\item Finally, a finite distributive lattice $D$ is \emph{fully chain-representable} if 
for  every $Q$ satisfying \eqref{eqTzhjP}, the inclusion  $Q\subseteq D$ is chain-representable.
\end{enumeratei}
\end{definition}

Armed with this definition, we formulate the following statement.

\begin{proposition}\label{propcLzGb} Let $D$ be a finite distributive lattice. Then $D$ is fully chain-representable if and only if it is planar and it has at most one join-reducible coatom.
\end{proposition}

Although Proposition~\ref{propcLzGb} is now a consequence of the conjunction of Cz\'edli~\cite{czgDQprincrep} and Gr\"atzer~\cite{ggwith1}, both \cite{czgDQprincrep} and \cite{ggwith1} contain long proofs. In the present paper, we give a direct and short  proof of Proposition~\ref{propcLzGb}. 

The following corollary will easily be concluded from the previous statements.

\begin{corollary}\label{corolsrpBngN} If a finite distributive lattice is fully \aarep-representable, then it is fully \flrep-representable.
\end{corollary}

Next, we collect some  known facts; most of them will be needed in our proofs. 
\semmi{Since there are several concepts of representability in the present paper, we use our terminology.} 

\begin{theorem}[Summarizing known facts]\label{thmsofar} Let $D$ be finite distributive lattice and assume that $Q\subseteq D$ satisfies \eqref{eqTzhjP}. Then 
 the following five statements hold.
\begin{enumeratei}
\item\label{thmsofara} \textup{(Gr\"atzer and Lakser~\cite{gratzerlakser}\footnote{Since I am mentioned in the addendum of \cite{gratzerlakser}, let me note that only some optimization of G.\ Gr\"atzer and H.\ Lakser's original proof of this fact is due to me.})} If $D$ is fully \flrep-representable,  then it is planar. 
\item\label{thmsofarb} \textup{(G.\ Gr\"atzer, personal communication)} If  $D$ is fully chain-rep\-resent\-able, then it is planar. 
\item\label{thmsofarc}  \textup{(Gr\"atzer~\cite{ggwith1})} If the inclusion $Q\subseteq D$ is principal congruence representable by a finite lattice, then it is chain-representable.
\item\label{thmsofard}  \textup{(Gr\"atzer~\cite{ggwith1})} If $D$ is fully \flrep-representable, then it is fully chain-represent\-able.
\item\label{thmsofare}  \textup{(Gr\"atzer~\cite{ggwith1})} If the inclusion $Q\subseteq D$ is chain-representable and $1_D\in J(D)$, then  this inclusion is principal congruence representable by a finite lattice.
\end{enumeratei}
\end{theorem}

Note that \eqref{thmsofard} is a particular case of \eqref{thmsofarc}, \eqref{thmsofare} is very deep, and the proof of \eqref{thmsofarb} is similar to that of \eqref{thmsofara}. Note also that there are several results on the representability of an  ordered set $Q$ as $\Princ(L)$ (without taking care of $D$), see
Cz\'edli~\cite{czgonemap}, \cite{czgaleph0}, \cite{czgprincout}, \cite{czgmanymaps}, and \cite{czgcometic} and    Gr\"atzer~\cite{gG13}, \cite{gGasm2016}, \cite{gGprincII}, and \cite{gGprincIII}. For more about full \pcrely, see 
Gr\"atzer~\cite{ggwith1} and  Gr\"atzer and Lakser~\cite{gratzerlakser}.

While Theorem~\ref{thmmain} and Remark~\ref{remarktLpR} give a satisfactory description of \emph{full} representability, we know much less on the representability of a \emph{single} inclusion $Q\subseteq D$. Theorem~\ref{thmsofar}\eqref{thmsofarc} and \eqref{thmsofare}, taken from  Gr\"atzer~\cite{ggwith1}, reduces the problem to chain-representability, provided that $1_D\in Q$. If 
$1_D\in Q$ is not assumed then 
we can prove only the following statement. 

\begin{proposition}\label{propositionbchrsnlsgR} If $Q$ is a subset of a finite distributive lattice $D$ such that the inclusion $Q\subseteq D$ is chain-representable, then this inclusion  is principal congruence representable by a finite algebra $A$. Furthermore, if $Q\subseteq D$  is represented by a $J(D)$-colored chain  $\tripl C\szn D$, then $A$ can be chosen so that $|A|=|C|$.
\end{proposition}

Note that, in the $1_D\in Q$ case, the finite lattice constructed in Gr\"atzer~\cite{ggwith1} to represent $Q\subseteq D$ has much more elements than $|C|$.

\color{black}

\subsection*{Outline of the rest of the paper}
We recall or prove some technical and mostly folkloric statements on planar distributive lattices in Section~\ref{sectionpropplandist}. Section~\ref{sectionchainrepr} contains the proof of Proposition~\ref{propcLzGb}. In Section~\ref{sectionconA}, we deal with congruences of arbitrary algebras and, with the exception of Proposition~\ref{sectionSnglPrf},  prove the rest of our statements formulated in the present section. 
{Finally, 
Section~\ref{sectionSnglPrf} is devoted to the proof of Proposition~\ref{sectionSnglPrf}.}

\begin{figure}[ht] 
\centerline
{\includegraphics[scale=1.0]{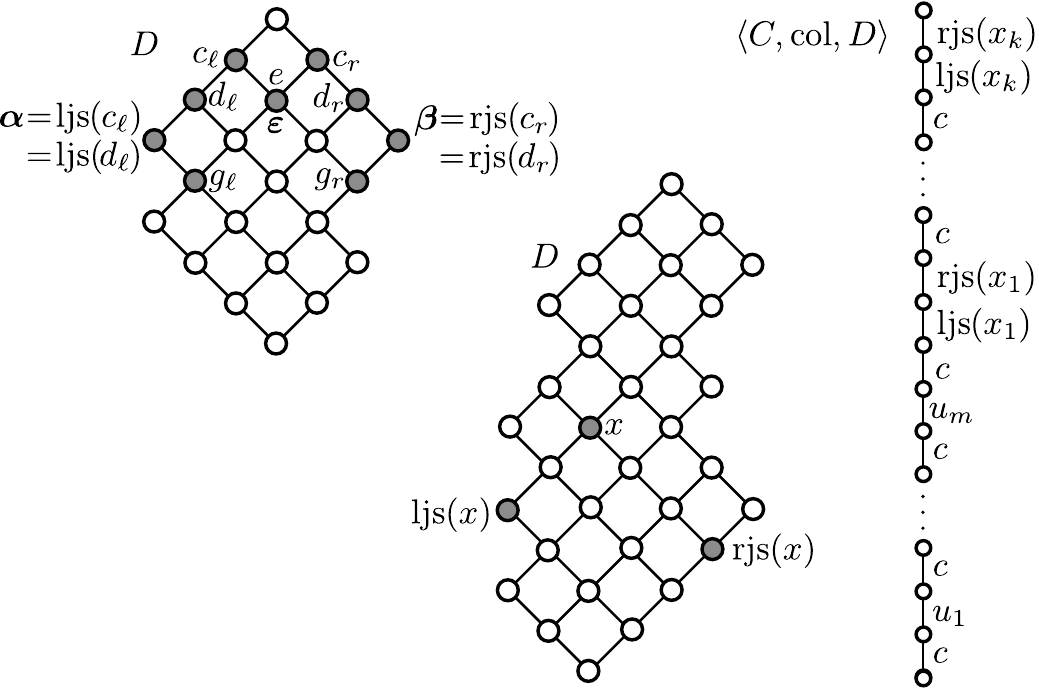}}
\caption{{Examples}}
\label{figexone}
\end{figure}%

\section{Properties of planar distributive lattices}
\label{sectionpropplandist}
In this section, we recall some  known facts about planar distributive lattices.
For concepts not defined here, see Cz\'edli and Gr\"atzer~\cite{czgggchapter} and see also the monographs Gr\"atzer~\cite{ggGLT} and  \cite{ggCFL2}. As far as \emph{distributive} lattices are considered, many of the facts below belong to the folklore. Interestingly, a lot of them are valid \emph{not only for distributive} lattices.  
In the whole section, unless otherwise stated, $D$ denotes a \emph{planar distributive lattice}.  Note that a planar lattice is \emph{finite} by definition.
It belongs to the folklore, see also Gr\"atzer and Knapp~\cite{gGknappI}, that 
\begin{equation}
\text{each element of $D$ has at most two covers and at most two lower covers.}
\label{eqtxt2cov}
\end{equation} 
As usual, we fix a \emph{planar} diagram of $D$; adjectives like ``left'' and ``right'' are understood modulo this diagram.  Assume that
\begin{equation}
\parbox{11cm}{$D$ has two distinct join-reducible coatoms, $c_\ell$ and $c_r$, and let $e=c_\ell\wedge c_r$;}
\label{eqtxtdzhGtrT} 
\end{equation}
the notation is chosen so that  $c_\ell$ is to the left of $c_r$; 
see the first diagram in Figure~\ref{figexone}. 
It follows from  \eqref{eqtxt2cov} that  $c_\ell$ and $c_r$ belong to the \emph{left boundary chain} $\lbound(D)$  and to the \emph{right boundary chain} $\rbound(D)$, respectively. 
Since $c_\ell$ is incomparable with $c_r$, these elements are not the least elements of the corresponding boundary chains. Hence, there are a unique $d_\ell\in\lbound(D)$ and a unique $d_r\in\rbound(D)$ such that $d_\ell\prec c_\ell$ and $d_r\prec c_r$.  
As usual, for $x\in D$,  the ideal $\set{y: y\leq x}$ will be denoted by $\ideal x$.  The following observation belongs to the folklore; note that it holds  in every slim semimodular\footnote{This concept, introduced in Gr\"atzer and Knapp~\cite{gGknappI} and
surveyed in Cz\'edli and Gr\"atzer~\cite{czgggchapter}, is not needed here; it suffices to know that every planar distributive lattice is slim and semimodular.} lattice, not only in our $D$. 
\begin{equation}
\parbox{10 cm}{For $x\in D$, the largest element  of 
 $\lbound (D)\cap J_0(D)\cap\ideal x$ is the \emph{left join support} of $x$; it is denoted by $\ljsp(x)$. The \emph{right join support} $\rjsp(x)$ of $x$ is the largest element of  $\rbound (D)\cap J_0(D)\cap\ideal x$. 
 Both $\ljsp(x)$ and $\rjsp(x)$ are join-irreducible elements (by definition) and we have that $x=\ljsp(x)\vee\rjsp(x)$;}
\label{eqtmGhGha}
\end{equation}
see the second part of Figure~\ref{figexone}. 
For the slim semimodular case, the equality in \eqref{eqtmGhGha} follows from the inclusion 
\begin{equation}
J(D)\subseteq \lbound (D)\cup\rbound (D); 
\label{eqdmcvhgbGbn}
\end{equation}
see Cz\'edli and Schmidt~\cite[Lemma 6]{czgschtvisualI} or Cz\'edli and Gr\"atzer~\cite{czgggchapter}. 
The following lemma is illustrated by the first part of Figure~\ref{figexone}.

\begin{lemma}\label{lemmaeqdczGnBw} Let $D$ be a finite distributive lattice satisfying  \eqref{eqtxtdzhGtrT}. Then, with the  notations from \eqref{eqtxtdzhGtrT} to \eqref{eqtmGhGha}, $1=\ljsp(d_\ell)\vee\rjsp(d_r)$.
\end{lemma}

\begin{proof} 
The  following  auxiliary statement is a transcript of Cz\'edli~\cite[Lemma 4.4]{czgdiagr}.
\begin{align}
\parbox{6.5cm}{For $x,y\in D$, 
$x$ is strictly to the left of $y$ iff $\ljsp(x)>\ljsp(y)$ and $\rjsp(x)<\rjsp(y)$.}\label{eqtmGhGhb}
\end{align}
By distributivity (in fact, by lower semimodularity), 
\begin{equation}
\text{$e\prec c_\ell$ and $e\prec c_r$.}
\label{eqtxtdtWcRhm}
\end{equation}
We claim that $d_\ell\neq e$. Suppose to the contrary that  $d_\ell=e$. 
Then, since $c_\ell\notin J(D)$ and $e=d_\ell$ belongs to $\lbound(D)$, we obtain that $c_\ell$ has a lower cover $u$ strictly to the right of $e$. By Kelly and Rival~\cite[Proposition 1.6]{kellyrival}, $e$ is strictly on the left and $c_r$ is strictly on the right of a maximal chain through $\set{u,c_\ell}$. This contradicts 
Kelly and Rival~\cite[Lemma 1.2]{kellyrival}
 since $e\prec c_r$.
Hence, $d_\ell\neq e$ and $e$ is strictly to the right of $d_\ell$. Thus, since $c_\ell\in\lbound(D)$ has a lower cover belonging to $\lbound(D)$ and it has at most two lower covers by \eqref{eqtxt2cov}, it follows that 
\begin{equation}
\parbox{9.5 cm}{$e$ is strictly to the right of $d_\ell$ and   $d_\ell\in \lbound(D)$. Similarly, $e$ is strictly to the left of $d_r$ and $d_r\in\rbound(D)$;}
\label{eqtxtlhzsmszmHT}
\end{equation}
see the first part of Figure~\ref{figexone}. 
Next, observe that $\ljsp(d_\ell)> \ljsp(e)$ and  $\rjsp(d_r)> \rjsp(e)$   by \eqref{eqtmGhGhb}. So, by the definition of left and right join supports,  
\begin{equation}
\ljsp(d_\ell)\nleq e\quad\text{and}\quad \rjsp(d_r)\nleq e.
\label{eqdbzfnrgTgH}
\end{equation}
Hence $e<\ljsp(d_\ell)\vee e \leq d_\ell \vee e  \leq c_\ell$, whereby \eqref{eqtxtdtWcRhm} gives that $\ljsp(d_\ell)\vee e=c_\ell$. Similarly,   $\rjsp(d_r)\vee e=c_r$. Using the equality from  \eqref{eqtmGhGha}
and the facts mentioned in the present paragraph, we obtain that 
\begin{equation*}
\begin{aligned}
&\ljsp(d_\ell)\vee\rjsp(d_r)=\ljsp(d_\ell)\vee \ljsp(e)\vee  \rjsp(e)\vee  \rjsp(d_r)
\cr
&\overset{\eqref{eqtmGhGha}}= \ljsp(d_\ell)\vee e\vee   \rjsp(d_r)= \bigl(\ljsp(d_\ell)\vee e\bigr)\vee \bigl(\rjsp(d_r)\vee e\bigr)
=c_\ell\vee c_r=1.
\end{aligned}
\label{eqdwJRbhT}
\end{equation*}
This completes the proof of Lemma~\ref{lemmaeqdczGnBw}.
\end{proof}

For later reference, note the following well-known consequence of distributivity:
\begin{equation}
(u\in J(D)\text{ and }u\leq v_0\jj\dots\jj v_{m-1})\Longrightarrow (\exists i<m)(u\leq v_i).
\label{eqdjlfkWsDbF}
\end{equation}

Next, we state and prove a technical lemma.

\begin{lemma}\label{lemmacqzTrGmN} Let $D$ be a finite distributive lattice satisfying  \eqref{eqtxtdzhGtrT}. Then the following four assertions hold.
\begin{enumeratei}
\item\label{lemmacqzTrGmNa} $J(D)\setminus\ideal e=\set{\ljsp(c_\ell),\rjsp(c_r)}$ and $\ljsp(c_\ell)\neq\rjsp(c_r)$.
\item\label{lemmacqzTrGmNb} $\set{\ljsp(c_\ell),\rjsp(c_r)}$ is the  set of maximal elements of $J(D)$ and $\ljsp(c_\ell)\parallel\rjsp(c_r)$.
\item\label{lemmacqzTrGmNc} $\ljsp(c_\ell)\nleq e\vee \rjsp(c_r)$ and  $\rjsp(c_r)\nleq e\vee \ljsp(c_\ell)$.
\item\label{lemmacqzTrGmNd} $e\nleq \ljsp(c_\ell)$ and  $e\nleq \rjsp(c_r)$.
\end{enumeratei}
\end{lemma}

\begin{proof}  We use the notation from  \eqref{eqtxtdzhGtrT} to \eqref{eqtmGhGha}. By \eqref{eqtxtdzhGtrT},   $c_\ell\notin J_0(D)$. Hence, $d_\ell\prec c_\ell$ yields that $\lbound(D)\cap J_0(D)\cap\ideal c_\ell=\lbound(D)\cap J_0(D)\cap\ideal d_\ell$. Thus,
\begin{equation}
\text{$\ljsp(c_\ell)=\ljsp(d_\ell)$. We obtain similarly that $\rjsp(c_r)=\rjsp(d_r)$.}
\label{eqtxtdbhzrnwGhxP}
\end{equation}
Lemma~\ref{lemmaeqdczGnBw} and  \eqref{eqtxtdbhzrnwGhxP} gives that $1 =\ljsp(c_\ell)\vee\rjsp(c_r)$, whereby  $\ljsp(c_\ell)\parallel \rjsp(c_r)$. So,
since  $J(D)\subseteq \ideal 1=\ideal(\ljsp(c_\ell)\vee\rjsp(c_r))$,  \eqref{eqdjlfkWsDbF} implies  \ref{lemmacqzTrGmN}\eqref{lemmacqzTrGmNb}. 
If we had that $e\leq \ljsp(c_\ell)$, then \eqref{eqtxtdbhzrnwGhxP}  and $\ljsp(d_\ell)\leq d_\ell$  would lead to $e\leq d_\ell$, contradicting  \eqref{eqtxtlhzsmszmHT}. Hence, 
$e\nleq \ljsp(c_\ell)$ and, similarly, $e\nleq \rjsp(c_r)$. This proves \ref{lemmacqzTrGmN}\eqref{lemmacqzTrGmNd}.
For the sake of contradiction, suppose that  \ref{lemmacqzTrGmN}\eqref{lemmacqzTrGmNc} fails. Let, say, $\ljsp(c_\ell)\leq e\vee \rjsp(c_r)$.  Thus, since $\ljsp(c_\ell)=\ljsp(d_\ell)\nleq e$ by \eqref{eqdbzfnrgTgH} and \eqref{eqtxtdbhzrnwGhxP}, we obtain by  \eqref{eqdjlfkWsDbF} and \eqref{eqtxtdbhzrnwGhxP} that 
$\ljsp(d_\ell)=\ljsp(c_\ell)\leq \rjsp(c_r)=\rjsp(d_r)$, contradicting Lemma~\ref{lemmaeqdczGnBw}. Hence, \ref{lemmacqzTrGmN}\eqref{lemmacqzTrGmNc}  holds.
Finally,  \ref{lemmacqzTrGmN}\eqref{lemmacqzTrGmNc} and \ref{lemmacqzTrGmN}\eqref{lemmacqzTrGmNd} give that $e\parallel \ljsp(c_\ell)$, whereby $e<e\vee \ljsp(c_\ell)\leq c_\ell$. So  \eqref{eqtxtdtWcRhm}  gives that 
$e\prec c_\ell =e\vee \ljsp(c_\ell)$. By distributivity (in fact, by lower semimodularity), $g_\ell:=e\wedge \ljsp(c_\ell)\prec  \ljsp(c_\ell)$. Similarly, $g_r:=e\wedge \rjsp(c_r)\prec  \rjsp(c_r)$.  Combining these covering relations with 
 \ref{lemmacqzTrGmN}\eqref{lemmacqzTrGmNb}, we obtain that
\begin{equation}
\parbox{10cm}{With the exception of $\ljsp(c_\ell)$, all join-irreducible elements of $\lbound(D)$ are in $\ideal e$. Similarly for $\rjsp(c_r)$ and $\rbound(D)$.}
\label{eqpbxdzkKknPws}
\end{equation}
Clearly, \ref{lemmacqzTrGmN}\eqref{lemmacqzTrGmNa} follows from   \eqref{eqdmcvhgbGbn} and \eqref{eqpbxdzkKknPws}.
\end{proof}

\section{Chain-representability}\label{sectionchainrepr}
The first paragraph in the following proof is based on G.\ Gr\"atzer's idea; see Theorem~\ref{thmsofar}\eqref{thmsofarb}.

\begin{proof}[Proof of Proposition~\ref{propcLzGb}]\ 
In order to prove the necessity of the condition formulated in the proposition, assume that $D$ is fully chain-representable. 
For the sake of contradiction, suppose that $D$ is not planar. It is well known from, say, Lemma 3-4.1 and the paragraph preceding it in Cz\'edli and Gr\"atzer~\cite{czgggchapter} that 
\begin{equation}
\parbox{10cm}{every non-planar finite distributive lattice contains a three-element antichain that consists of join-irreducible elements.}
\label{eqpbxthmntChn}
\end{equation}
Thus, we can pick a three-element antichain $\set{p_0,p_1,p_2}\subseteq J(D)$.
 Let $p=p_0\vee p_1\vee p_2$ and  $Q=J^+(D)\cup \set{p}$. Since $D$ is fully chain-representable, we can take a $J(D)$-colored chain $\tripl C\szn D$ such that 
$Q=\srep(C,\szn,D)$. 
Take a maximal interval
\[[x,y]= \set{x=x_0 \prec x_1\prec \dots\prec x_n=y}
\]
in $C$ such that $\erep([x,y])=p$. 
Letting $r_i:=\szn([x_{i},x_{i+1}])\in J(D)$, 
we have that $\bigvee_{i<n} r_i=p=p_0\vee p_1\vee p_2$.
By \eqref{eqdjlfkWsDbF}, each of $p_0$, $p_1$, and $p_2$ is less than or equal to some of the $r_i$, $i<n$. By the same reason, each of the $r_i$, $i<n$,  is less than or equal to some of $p_0$, $p_1$, and $p_2$. Therefore, since $\set{p_0,p_1,p_2}$ is an antichain,
each of  $p_0$, $p_1$, and $p_2$ occurs among the $r_i$,  $i<n$. Without loss of generality, we can assume that $p_0$ or $p_2$ occurs before $p_1$.  
Take a maximal subinterval $[x_i,x_j]$ of $[x,y]$ such that 
$\erep([x_i,x_{j}]) = p_1$; we have that $0<i$ since 
$p_0$ or $p_2$ occurs before $p_1$. Then $r_{i-1}=\szn([x_{i-1},x_{i}])\nleq p_1$, whereby 
\[p_1<r_{i-1}\vee p_1= \erep([x_{i-1},x_{j}])\in \srep(C,\phi) = Q.
\]
Since $r_{i-1}$ is less than or equal to some of $p_0$, $p_1$, and $p_2$ but $r_{i-1}\nleq p_1$, we can assume that $r_{i-1}\leq p_0$. Thus, since $\set{p_0,p_1,p_2}$ is an antichain, $r_{i-1} \neq r_{i-1}\vee p_1$. Hence, 
$r_{i-1}\vee p_1$ is join-reducible, and the choice
of $Q$ gives that $r_{i-1}\vee p_1\in\set{p,1}$.
So $p_2\leq r_{i-1}\vee p_1$, and  \eqref{eqdjlfkWsDbF} yields that $p_2\leq p_1$ or $p_2\leq r_{i-1}\leq p_0$,
which contradicts the fact that $\set{p_0,p_1,p_2}$ is an antichain. This proves that $D$ is planar.  

Next, striving for a contradiction,  suppose  that \eqref{eqtxtdzhGtrT} holds. With the notation given from \eqref{eqtxtdzhGtrT} to \eqref{eqtmGhGha},  let $Q:=J^+(D)\cup\set{e}$; this need not be the same $Q$ as above. 
Since there are two coatoms,  $1_D\notin J(D)$. Since $D$ is fully chain-representable, there is a 
$J(D)$-colored chain $\tripl C\szn D$ such that $Q=\srep(C,\szn,D)$. Since $e\in Q$, there is a maximal interval $[x,y]$ in $C$ such that $e=\erep([x,y])$. Since $e\neq 1_D=\erep(C)$, either $0_C<x$, or $y<1_C$;
by duality, we can assume that $y<1_C$. 
Let $z$ be the cover of $y$ in $C$, and let  
$p:=\szn([y,z])\in J(D)$.
 By the maximality of $[x,y]$, we have that $e<e\vee p$. Since there are only three elements, $1=c_\ell\vee c_r$, $c_\ell=d_\ell \vee e$, and $c_r=d_r\vee e$ strictly above $e$ and they  are join-reducible, $e\vee p\notin J(D)$. Thus, since $e\vee p=\erep([x,z])\in \srep(C,\szn, D)=Q$, it follows that $e\vee p=1$. Hence, $\ljsp(d_\ell)\leq e\vee p$, whereby \eqref{eqdbzfnrgTgH} and \eqref{eqdjlfkWsDbF} give that $\ljsp(d_\ell)\leq p$. Since $\rjsp(d_r)\leq p$ follows similarly, Lemma~\ref{lemmaeqdczGnBw} gives that $1_D=p\in J(D)$.  
This is a  contradiction,  proving the 
necessity part of Proposition~\ref{propcLzGb}.

In order to prove the sufficiency part, assume that $D$ is planar and it has at most one join-reducible coatom.
Let $Q\subseteq D$ such that \eqref{eqTzhjP} holds.
We need to find a $J(D)$-colored chain that represents $Q$.  We can assume that $|D|\geq 2$ since otherwise the one-element $J(D)$-colored chain represents $Q$. 

If $1\in J(D)$, then we let $c:=1$; then the principal filter $\filter c=\set{x\in D: c\leq x}$ is clearly a subset of $Q$. If $1\notin J_0(D)$,  then  there are exactly two coatoms by \eqref{eqtxt2cov} and at least one of them is join-irreducible by our assumption; in this case, let $c$ be a join-irreducible coatom and we still have that $\filter c\subseteq Q$.

Let $u_1,u_2,\dots, u_m$ be a repetition-free list of all elements of $J(D)$. Similarly, let $x_1,x_2,\dots, x_k$ be a repetition-free list of all elements of $Q\setminus J(D)$; this list can be empty. 
Define $\tripl C\szn D$ such that  $\length(C)= 2m+3k-1$ and the colors of the edges,  from bottom to top, are as follows:
\begin{equation*}
\begin{aligned}
 u_1,{} &{}c, u_2,\dots, c, u_m,
c , \ljsp(x_1), \rjsp(x_1),
c ,\ljsp(x_2), \rjsp(x_2),c,
\cr
&
\ljsp(x_3), \rjsp(x_3), 
c ,\ljsp(x_4), \rjsp(x_4),c,
\dots,
\ljsp(x_k), \rjsp(x_k);
\end{aligned}
\end{equation*}
see the third part of Figure~\ref{figexone} where  the map  $\szn$ is given by labeling. The equation in \eqref{eqtmGhGha} makes it clear that $Q\subseteq \srep(C,\szn,D)$. In order to see the converse inclusion, take an arbitrary element of $\srep(C,\szn,D)$, that is, take an arbitrary $I\in\Intv(C)$ and consider $\erep(I)$; we need to show that $\erep(I)\in Q$. This is clear if $\length(I)\leq 1$.  For $\length(I)\geq 2$, either
$\erep(I)$ is in the filter $\filter c$,
which is a subset of $Q$, or $\erep(I)= \ljsp(x_i)\vee\rjsp(x_i)$ and \eqref{eqtmGhGha} gives that $\erep(I)=x_i\in Q$. Therefore, $Q=\srep(C,\szn,D)$, as required.
This completes the proof of Proposition~\ref{propcLzGb}.
\end{proof}

\section{Dealing with congruences of an algebra}\label{sectionconA} 

\begin{lemma}\label{lemmadznPbWqbns} 
Every  fully \aaerep-representable finite distributive lattice is planar.
\end{lemma}

\begin{proof} For the sake of contradiction, suppose that $D$ is a \emph{non-planar} fully \aaerep-representable finite distributive lattice. Pick a three-element antichain $\set{\balpha,\bbeta,\bgamma}\subseteq J(D)$; see \eqref{eqpbxthmntChn}. Let $\bmu=\balpha\vee\bbeta\vee\bgamma$ and $Q=J^+(D)\cup\set\bmu$. By the full \aaerep-representability of $D$, we can assume that $D=\Con(A)$ and $Q=\Princ(A)$ for an algebra $A$. Then $\bmu$ is a principal congruence of $A$, whereby we can pick elements $a,b\in A$ such that $\con(a,b)$, the \emph{principal congruence generated} by the pair $\pair a b$, is $\bmu$; in notation, $\bmu=\con(a,b)$. Using that $\bmu=\balpha\vee\bbeta\vee\bgamma$, we can pick a \emph{shortest} sequence $a=x_0$, $x_1$, \dots, $x_{n-1}$, $x_n=b$ of elements of $A$ such that $\pair{x_i}{x_{i+1}}$ belongs to $\balpha\cup\bbeta\cup\bgamma$ for all (non-negative) $i<n$. That is, for all $i<n$, 
\begin{equation}
\con(x_i, x_{i+1})\leq \balpha,\quad \con(x_i, x_{i+1})\leq \bbeta,\quad\text{or}\quad \con(x_i, x_{i+1})\leq \bgamma.
\label{eqfPzltmmS}
\end{equation}
We claim that
\begin{equation}
\text{each of $\balpha$, $\bbeta$, and $\bgamma$ occurs in \eqref{eqfPzltmmS} for some $i<n$.}
\label{eqtxtChclcRs}
\end{equation}
Suppose to the contrary that, say, $\bgamma$ does not occur. Then $\bgamma\leq\bmu=\balpha\vee\bbeta$ and \eqref{eqdjlfkWsDbF} give that $\bgamma\leq\balpha$ or $\bgamma\leq\bbeta$, which is impossible since $\set{\balpha,\bbeta,\bgamma}$ is an antichain. This shows the validity of \eqref{eqtxtChclcRs}, and shows also that $\balpha\vee \bbeta<\bmu$.

Our sequence yields that $\con(a,b)\leq \bigvee_{i<n}\con(x_i, x_{i+1})$. Thus, 
$\balpha\leq \bmu=\con(a,b)\leq \bigvee_{i<n}\con(x_i, x_{i+1})$. 
Hence, \eqref{eqdjlfkWsDbF} gives an $i_\balpha<n$ such that $\balpha\leq \con(x_{i_\balpha}, x_{i_\balpha+1})$. Combining this inequality with \eqref{eqfPzltmmS} and taking into account that
$\set{\balpha,\bbeta,\bgamma}$ is an antichain, we obtain that  
$\balpha = \con(x_{i_\balpha},  x_{i_\balpha+1})$. Similarly,  $\bbeta = \con(x_{i_\bbeta},  x_{i_\bbeta+1})$ and $\bgamma = \con(x_{i_\bgamma},  x_{i_\bgamma+1})$ for some $i_\bbeta< n$ and $i_\bgamma<n$.
Since $\balpha$, $\bbeta$, and $\bgamma$ play a symmetric role, we can assume that $i_\balpha$
is between $i_\bbeta$ and $i_\bgamma$. Let $j$ denote the smallest non-negative number such that $j\leq i_\balpha$ and 
$\con(x_{s},  x_{s+1})\leq \balpha$
for all $s\in [j,i_\balpha]:=\set{j,j+1,\dots, i_\balpha}$. By $\con(x_{i_\balpha},  x_{i_\balpha+1})=\balpha$, this $j$ exists.
Since  $\con(x_{i_\bbeta},  x_{i_\bbeta+1})=\bbeta\nleq\balpha$,  $\con(x_{i_\bgamma},  x_{i_\bgamma+1})=\bgamma\nleq\balpha$,  and  $i_\balpha$
is between $i_\bbeta$ and $i_\bgamma$, it follows that $j>0$. So we can consider 
$\con(x_{j-1},  x_{j})$, which is not in $\ideal\balpha:=\set{\xi\in D: \xi\leq\balpha}$.
By \eqref{eqfPzltmmS} and since the role of $\bbeta$ and $\bgamma$ is symmetric, we can assume that  
$\con(x_{j-1},  x_{j})\leq \bbeta$. 
The minimality of the length $n$  of our sequence implies that $x_{j-1}\neq x_{j+1}$. Hence 
\begin{equation} 0<\con(x_{j-1},x_{j+1})\leq \con(x_{j-1}, x_{j})\vee \con(x_{j}, x_{j+1})\leq \bbeta\vee\balpha<\bmu
\label{eqdkbzhNmYxZ}
\end{equation}
and the choice of $Q=\Princ(A)$ imply that the principal congruence $\con(x_{j-1},x_{j+1})$ is join-irreducible. Hence, applying \eqref{eqdjlfkWsDbF} to \eqref{eqdkbzhNmYxZ}, we obtain that $\con(x_{j-1},x_{j+1})\leq  \bbeta$ or $\con(x_{j-1},x_{j+1})\leq\balpha$. Thus, omitting $x_j$ from our sequence, we obtain a shorter sequence that still satisfies \eqref{eqfPzltmmS}. This contradicts the minimality of $n$ and proves the lemma.
\end{proof}

\begin{lemma}\label{lemmacxGnHjH}
Every  fully \aaerep-representable finite distributive lattice has at most one join-reducible coatom.
\end{lemma}

\begin{proof} For the sake of contradiction, suppose that $D$ is a fully \aaerep-representable finite distributive lattice that has at least two join-reducible coatoms.
By Lemma~\ref{lemmadznPbWqbns}, $D$ is planar. It follows from \eqref{eqtxt2cov} that  $D$ has exactly two coatoms, whereby  \eqref{eqtxtdzhGtrT} holds. However, since we are going to compute with congruences, let $\epsilon=e$, $\balpha=\ljsp(c_\ell)$, and $\bbeta=\rjsp(c_r)$; see the diagram on the left of Figure~\ref{figexone}. Let $Q=J^+(D)\cup\set{\epsilon}$. The full \aaerep-representability of $D$ allows us to  assume that $D=\Con(A)$ and $Q=\Princ(A)$ hold for an algebra $A$. We claim that for every $u\neq v\in A$,
\begin{equation}
\parbox{9cm}{there exists a finite sequence $u=w_0, w_1,\dots, w_k=v$ of elements of $A$ such that $\con(w_i,w_{i+1})\in J(D)$ for all $i<k$.}
\label{eqpbxzhBnmXgrp}
\end{equation}
Note that this is valid for an arbitrary finite lattice $D$, not only for a planar distributive one.
We prove \eqref{eqpbxzhBnmXgrp} by induction on $\hei(\con(u,v))$, where $\hei$ is the \emph{height function} on $D$, that is, for $x\in D$, $\hei(x)$ is the length of $\ideal x$. Since $u\neq v$, the smallest possible value of $\hei(\con(u,v))$ is 1. So, to deal with the base of the induction, assume that $\hei(\con(u,v))=1$. Then $\con(u,v)$ is an atom, whereby $\con(u,v)\in J(D)$ and \eqref{eqpbxzhBnmXgrp} holds with $\tuple{k,w_0,\dots,w_k}:=\tuple{1,u,v}$.
Next, to perform the induction step, assume that $\hei(\con(u,v))>1$. We can also assume that $\con(u,v)\notin J(D)$, because otherwise 
 $\tuple{k,w_0,\dots,w_k}:=\tuple{1,u,v}$ 
does the job again. Then there are $\bgamma,\bdelta\in D=\Con(A)$ such that $\hei(\bgamma)<\hei(\con(u,v))$, $\hei(\bdelta)<\hei(\con(u,v))$, and $\con(u,v)=\bgamma\vee\bdelta$. Therefore, there is a \emph{shortest} sequence $u=s_0,s_1,\dots,s_{m-1}, s_m=v$ of elements of $A$ such that $\pair{s_j}{s_{j+1}}\in\bgamma\cup \bdelta$ for all $j<m$. That is, for all $j<m$, $\con(s_j,s_{j+1})\leq\bgamma$ or $\con(s_j,s_{j+1})\leq\bdelta$. Thus, $\hei(\con(s_j,s_{j+1}))\leq\max\set{\hei(\bgamma),\hei(\bdelta)}<\hei(\con(u,v))$ for all $j<m$. Furthermore, $s_j\neq s_{j+1}$ for all $j<m$ since we took a shortest sequence. Hence, the induction hypothesis yields a sequence of type \eqref{eqpbxzhBnmXgrp} from $s_j$ to $s_{j+1}$ for all $j<m$. Concatenating these sequences, we obtain a sequence from $u$ to $v$ as  required by \eqref{eqpbxzhBnmXgrp}. This completes the induction and proves  \eqref{eqpbxzhBnmXgrp}. 

Next, we claim that 
\begin{equation}
\text{for every block $U$ of $\epsilon$, we have that $U^2\subseteq \balpha$ or $U^2\subseteq \bbeta$.}
\label{eqtxtznmrNpk}
\end{equation}
Since this is evident for a singleton $U$, assume that $|U|>1$. However, $U\neq A$ since $\epsilon\neq 1_D=1_{\Con(A)}$. So we an pick an element $x\in A\setminus U$ and another element $y\in U$. By \eqref{eqpbxzhBnmXgrp}, there is a finite sequence of elements from $x$ to $y$ such that any two  consecutive elements in this sequence generate a join-irreducible congruence. This sequence begins outside $U$ and terminates in $U$, whereby there are two consecutive elements in the sequence such that first is outside $U$ but the second is in $U$. Changing the notation if necessary, we can assume that these two elements are $x$ and $y$. Thus, $x\in A\setminus U$ and $y\in U$ are chosen so that 
$\con(x,y)\in J(D)$. It follows from Lemma~\ref{lemmacqzTrGmN}\eqref{lemmacqzTrGmNb} that 
$\con(x,y)\leq\balpha$ or $\con(x,y)\leq\bbeta$. Since the role of $\balpha$ and $\bbeta$ is symmetric, we can assume that $\con(x,y)\leq\balpha$.   Now let $z$ be an arbitrary element of $U$. Since $\con(x,z)$ is a principal congruence, the choice of $Q=\Princ(A)$ and $x\neq z$ give that $\con(x,z)\in \set{1_D,\epsilon}\cup J(D)$. 
If we had that 
$\con(x,z)=1_D=1_{\Con(A}$, then $\bbeta\leq 1_D=\con(x,z)\leq \con(x,y)\vee \con(y,z) \leq \balpha\vee \epsilon$, which would contradict Lemma~\ref{lemmacqzTrGmN}\eqref{lemmacqzTrGmNc}. Since $z$ is in $U$ but $x$ is not, $\con(x,z)\notin\ideal\epsilon$. In particular, $\con(x,z)\neq \epsilon$ and we obtain that $\con(x,z)\in J(D)$. In fact, $\con(x,z)\in J(D)\setminus\ideal \epsilon$, and it follows  from  Lemma~\ref{lemmacqzTrGmN}\eqref{lemmacqzTrGmNa} that $\con(x,y)$ is $\balpha$ or $\bbeta$. 
If we had that $\con(x,z)$ is $\bbeta$, then $\bbeta=\con(x,z)\leq \con(x,y)\vee\con(y,z)\leq \balpha\vee \epsilon$ would contradict Lemma~\ref{lemmacqzTrGmN}\eqref{lemmacqzTrGmNc}. Hence, $\con(x,z)=\balpha$ and $\pair yz\in \con(y,x)\vee\con(x,z)= \con(x,y)\vee\balpha=\balpha$.
Since $z$ was an arbitrary element of $U$, the required inclusion $U^2\subseteq \balpha$ follows by the transitivity of $\balpha$. This proves \eqref{eqtxtznmrNpk}.

Finally, since $\epsilon\in Q=\Princ(A)$,
we can pick $a,b\in A$ such that $\epsilon=\con(a,b)$. Clearly, the $\epsilon$-block of $a$ contains $b$. Applying \eqref{eqtxtznmrNpk} to this block, it follows that $\pair a b\in\balpha$ or $\pair a b\in\bbeta$. Thus, $\epsilon=\con(a,b)\leq\balpha$ or $\epsilon\leq\bbeta$, which contradicts Lemma~\ref{lemmacqzTrGmN}\eqref{lemmacqzTrGmNd}. This completes the proof of Lemma~\ref{lemmacxGnHjH}.
\end{proof}

Now, we are in the position to prove our main theorem.

\begin{proof}[Proof of Theorem~\ref{thmmain}]
We need to prove only that \ref{thmmain}\eqref{thmmaina} implies  \ref{thmmain}\eqref{thmmaind}; this follows from Lemmas~\ref{lemmadznPbWqbns} and \ref{lemmacxGnHjH}.
\end{proof}

The following proof relies heavily on Gr\"atzer~\cite{ggwith1}.

\begin{proof}[Proof of Proposition~\ref{propdhmhZ}] Let $D$ be finite distributive lattice; we can assume that $|D|>1$.
In order to prove that (a) implies (b), assume that $D$ is fully \aarep-representable. Then it is fully \aaerep-representable, and it follows from Theorem~\ref{thmmain} that $D$ is planar. We obtain from \eqref{eqtxt2cov} that $D$ has at most two coatoms, and we need to exclude the possibility that $D$ has exactly two coatoms. 
\begin{equation}
\parbox{6.3cm}{Suppose to the contrary that $D$ has two coatoms, $\balpha$ an $\bbeta$, and let $Q=D\setminus\set{1_D}$.}
\label{eqdkbhkzHbW}
\end{equation}
We can assume that 
$D=\Con(A)$ and $Q=\Princ(A)$ for some algebra $A$, since  $D$ is fully \aarep-representable. We claim that 
\begin{equation}
1_{\Con(A)} =\balpha\cup \bbeta.
\label{eqcfhnhnTgH}
\end{equation}
In order to see this, let $\pair x y\in A^2$. Since $1_{\Con(A)}\notin\Princ(A)$, we have that $\con(x,y)\neq 1_{\Con(A)}=1_D$. Since $D$ has only two coatoms,  $\con(x,y)\leq\balpha$ or $\con(x,y)\leq\bbeta$.  This means that $\pair x y\in\balpha$ or $\pair x y\in\bbeta$, implying \eqref{eqcfhnhnTgH}.

Next, let $U$ be arbitrary $\balpha$-block. Since $\balpha\neq 1_D=1_{\Con(A)}$, we can pick an element $x\in A\setminus U$. For every $y\in U$, we have that $\pair x y\notin \balpha$ since $x$ is outside the $\balpha$-block $U$ of $y$. Hence, \eqref{eqcfhnhnTgH} gives that $\pair x y\in \bbeta$ for all $y\in U$. So $U^2\subseteq \bbeta$ by transitivity, and we conclude that $\balpha\leq \bbeta$. This is a contradiction since $\balpha$ and $\bbeta$ are distinct coatoms. Consequently, (a) implies (b). 

Next, in order to prove that (b) implies (a), assume that $D$ has exactly one coatom.  Let $Q$ be a subset of $D$ satisfying \eqref{eqAtgjWrSD}. Since $1_D\in J(D)$, we have that $J_0(D)=J^+(D)$, whereby $Q$ satisfies \eqref{eqTzhjP}.  Hence, Proposition~\ref{propcLzGb} implies that the inclusion $Q\subseteq D$ is chain-representable. Thus, Gr\"atzer~\cite[Theorem 3]{ggwith1}, which has been recalled in Theorem~\ref{thmsofar}\eqref{thmsofare}, gives that the inclusion in question 
is represented by the principal congruences of a finite lattice. Consequently, $D$ is  fully
\aarep-representable and condition \ref{propdhmhZ}(a) holds, as required.
\end{proof}

\begin{proof}[Proof of Corollary~\ref{crollarczhTngPr}] For the sake of contradiction, suppose that $A$ is an algebra such that $\Princ(A)=V$. Every congruence $\bgamma\in\Con(A)$ is the join of all principal congruences in $\ideal \bgamma$, whereby $D:=\Con(A)$ is the four-element boolean lattice. Thus, $A$ represents the inclusion $V:=Q\subseteq D$. 
By \eqref{eqdkbhkzHbW}, this is a contradiction.
\end{proof}

\begin{proof}[Proof of Corollary~\ref{corolsrpBngN}]
Assume that $D$ is fully \aarep-representable. 
We can also assume that $|D|>1$.
By  Proposition~\ref{propdhmhZ}, $D$ is planar and $1_D\in J(D)$. Proposition~\ref{propcLzGb} gives that $D$ is fully chain-representable. Hence, $D$ is fully \flrep-representable by Theorem~\ref{thmsofar}\eqref{thmsofare}.
\end{proof}


\section{Representing a single inclusion by a finite algebra}\label{sectionSnglPrf}
\color{black}
The aim of this section is to prove Proposition~\ref{propositionbchrsnlsgR}.

\begin{figure}[ht] 
\centerline
{\includegraphics[scale=1.0]{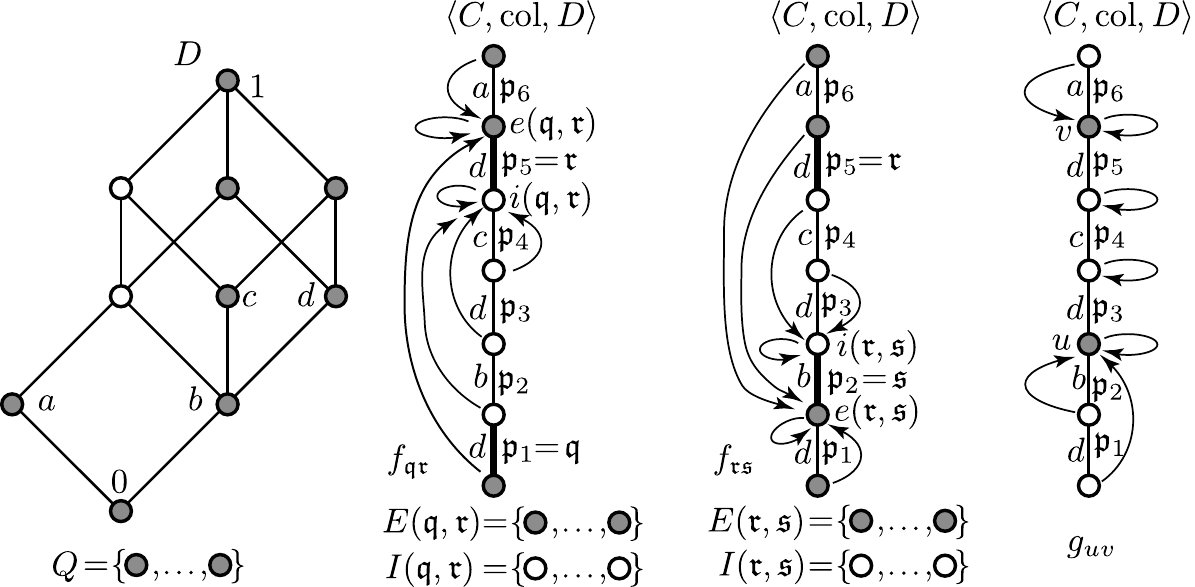}}
\caption{{From a $J(D)$-colored chain to an algebra}}
\label{figxtwo}
\end{figure}%

\begin{proof}[Proof of Proposition~\ref{propositionbchrsnlsgR}]
Assume that $D$ is a finite distributive lattice and 
a (finite) $J(D)$-colored chain  $\tripl C\szn D$ represents an inclusion $Q\subseteq D$. An example is given in Figure~\ref{figxtwo}, where $\tripl C\szn D$ is drawn thrice. 
In order to indicate generality, $|C|=7$ in the figure; note however that a five-element chain whose edges are colored with $c,b,d,a$, in this order, would also represent $Q\subseteq D$. 
In the figure, the colors are given by the labels $a$, \dots, $d$, and the edges of $C$, that is, the members of $\Prime(C)$, are $\inp_1,\dots, \inp_6$.  
In order to turn $C$ into an algebra, we are going to define two kinds of unary operations on $C$. 
First, for $u<v\in C$, we define a so-called \emph{contraction operation}  $g_{uv}\colon C\to C$ by the rule
\begin{equation*}
g_{uv}(x):= \begin{cases}
v,&\text{ if } x\geq v,\cr
x,&\text{ if } u\leq x\leq v,\text{ and}\cr
u,&\text{ if } x\leq u;
\end{cases}
\end{equation*}
see on the right of Figure~\ref{figxtwo}. The name  ``contraction'' comes from the straightforward fact that 
\begin{equation}
\parbox{10.5cm}{if $w<z\in C$ and $g_{uv}(w)\neq g_{uv}(z)$, then  $w\leq g_{uv}(w)< g_{uv}(z)\leq z$.}
\label{eqtxccngtrTn}
\end{equation}
Second, let 
$\inp,\inh\in \Prime(C)$ be \emph{distinct} edges of $C$. We define the \emph{interior set} $\belsh\inp\inh$ and the \emph{exterior set} $\kulsh\inp\inh$ as follows; see Figure~\ref{figxtwo} for  $\pair\inp\inh\in \set{\pair\inq\inr, \pair\inr\ins }$.
\begin{equation*}
\belsh\inp\inh:= \begin{cases}
[1_{\inp},0_{\inh}]&\text{ if }1_{\inp}\leq 0_{\inh},\cr
[1_{\inh},0_{\inp}]&\text{ if }1_{\inh}\leq 0_{\inp}
\end{cases}\qquad\text{and}\qquad \kulsh\inp\inh:=C\setminus \belsh\inp\inh.
\end{equation*} 
We also need to define the \emph{interior} and the   and the \emph{exterior target} elements $\belse\inp\inh$  and 
$\kulse\inp\inh$, respectively, as follows; see again the figure for  $\pair\inp\inh\in \set{\pair\inq\inr, \pair\inr\ins }$.
\begin{equation*}
\belse\inp\inh:= \begin{cases}
0_{\inh}&\text{ if }1_{\inp}\leq 0_{\inh},\cr
1_{\inh}&\text{ if }1_{\inh}\leq 0_{\inp}
\end{cases}
\qquad\text{and}\qquad
\kulse\inp\inh:= \begin{cases}
1_{\inh}&\text{ if }1_{\inp}\leq 0_{\inh},\cr
0_{\inh}&\text{ if }1_{\inh}\leq 0_{\inp}.
\end{cases} 
\end{equation*} 
With the notation given above, we define a unary \emph{forcing operation}
\begin{equation*}
f_{\inp\inh}\colon C\to C\qquad\text{by}\qquad
x\mapsto f_{\inp\inh}(x)
:= \begin{cases}
\belse\inp\inh &\text{ if } x\in \belsh\inp\inh,\cr
\kulse\inp\inh &\text{ if } x\in \kulsh\inp\inh.
\end{cases}
\end{equation*} 
The name ``forcing'' is motivated by the fact that the presence of this operation ``forces'' the inequality $\con(0_\inp,1_\inp)\geq \con(0_\inh,1_\inh)$. 
For $\pair\inp\inh\in \set{\pair\inq\inr, \pair\inr\ins }$, the forcing operation is  given in the middle part of Figure~\ref{figxtwo}.  
The unary $\alga$ algebra we need is defined by
\begin{equation}
\begin{aligned}
\alga=\tuple{C; \set{g_{uv}: u<v\in C}\cup 
\set{f_{\inp\inh}& : 
\inp\neq\inh\in\Prime(C) \cr
&\kern 5pt \text{ and }\szn(\inp)\geq\szn(\inh)\text{  holds in }D}}.
\end{aligned}
\label{eqalgDf}
\end{equation}
Although $\alga$ does not have a lattice reduct, we will frequently refer to the ordering of the chain $C$; for example, when speaking of intervals and edges. For $u,v\in C$, the \emph{symmetrized interval} $[u\wedge v, u\vee v]$ will be denoted by $\symi u v$. 
We claim that the map
\begin{equation}
\begin{aligned}
\phi\colon D&\to\Con(\alga),\quad\text{ defined by }\cr
x&\mapsto\set{\pair u v: \szn(\inp)\leq x\text{ for all } \inp\in \Prime(\symi u v)},
\end{aligned} \label{eqwdfPhhns}
\end{equation}
is a lattice isomorphism.

First, let $x\in D$; we need to show that $\phi(x)$ is a congruence of $\alga$. Since we use a symmetrized interval in \eqref{eqwdfPhhns}, $\phi(x)$ is reflexive and symmetric. The transitivity of $\phi(x)$ follows from the rule $\symi u w\subseteq \symi u v\cup \symi v w$. It is clear from \eqref{eqtxccngtrTn} that every contraction operation preserves $\phi(x)$. Let $f_{\inp\inh}$ be a forcing operation of $\alga$; this means that $\inp\neq\inh\in\Prime(C)$ and $\szn(\inp)\geq \szn(\inh)$ in $D$. In order to show that $f_{\inp\inh}$ preserves $\phi(x)$, assume that $\pair u v\in \phi(x)$ and $f_{\inp\inh}(u)\neq f_{\inp\inh}(v)$. Then 
$\set{f_{\inp\inh}(u),  f_{\inp\inh}(v)}=\set{0_\inh,1_\inh}$, whereby $\symi{f_{\inp\inh}(u)}{f_{\inp\inh}(v)}=\inh$. Hence, to show that 
$\tuple{f_{\inp\inh}(u),  f_{\inp\inh}(v)}\in\phi(x)$,
we need to show that $\szn(\inh)\leq x$. Since 
 $f_{\inp\inh}(u)\neq f_{\inp\inh}(v)$,  one of $u$ and $v$ is in $\kulsh\inp\inh$ while the other one is in $\belsh\inp\inh$. Hence, $u$ and $v$ in the chain $C$ are ``separated'' by at least one of the edges $\inp$ and $\inh$. If they are separated by $\inh$, then $\szn(\inh)\leq x$ by the definition of $\phi(x)$, as required. So we can assume that $u$ and $v$ are separated by $\inp$. Then $\szn(\inp)\leq x$ and, by \eqref{eqalgDf}, $\szn(\inp)\geq\szn(\inh)$. 
By transitivity, we obtain again that $\szn(\inh)\leq x$. This proves that $\phi(x)\in\Con(\alga)$, whereby \eqref{eqwdfPhhns} really defines a map from $D$ to $\Con(\alga)$.

Next, to prove the surjectivity of $\phi$, let $\Theta\in \Con(\alga)$. 
 Define  $x=\psi(\Theta)\in D$ by
\begin{equation}
x=\psi(\Theta):=\bigvee\set{\szn(\inp): \inp\in\Prime(C) \text{ and } \pair{0_\inp}{1_\inp}\in\Theta};
\label{eqdizTrNvq}
\end{equation} 
we are going to show that $\phi(x)=\Theta$. In order to do so, assume that $\pair w z\in \Theta$. Since both $\Theta$ and $\psi(x)$ are congruences of $\alga$, we can also assume that $w<z$. If $\inp\in \Prime([w,z])$, then 
$\pair{0_\inp}{ 1_\inp} = \pair{g_{0_\inp  1_\inp}(w)} {g_{0_\inp  1_\inp}(z)}\in\Theta$, whereby
$\szn(\inp)\leq x$. Since this holds for all $\inp\in\Prime(\symi w z )=\Prime([ w, z] )$, we obtain by \eqref{eqwdfPhhns} that $\pair w z\in\phi(x)$. Thus, $\Theta\leq \phi(x)$ holds in $\Con(\alga)$. Conversely, assume that $\pair w z\in\phi(x)$ and $w<z$.  Since $w<z$, 
\begin{equation}
\text{there are unique elements $t_i\in C$ 
such that $w=t_0\prec t_1\prec\dots\prec t_k=z$.}
\label{eqtxthtBfk}
\end{equation}
By \eqref{eqwdfPhhns},  $\szn([t_i,t_{i+1}])\leq x$ for all $i<k$. Thus, combining \eqref{eqdjlfkWsDbF} and \eqref{eqdizTrNvq}, we obtain an edge $\inp_i\in\Prime(C)$ such that $\szn([t_i,t_{i+1}])\leq \szn(\inp_i)$ and $\pair{0_{\inp_i}}{1_{\inp_i}}\in\Theta$, for every $i<k$. Applying the forcing operation  $f_{\inp_i\, [t_i,t_{i+1}]}$ componentwise to the pair $\pair{0_{\inp_i}}{1_{\inp_i}}$, we obtain $\pair{t_{i+1}}{t_{i}}$.  Since this operation preserves $\Theta$, we obtain that 
\begin{equation}
\pair{t_{i}}{t_{i+1}}\in \Theta,\quad\text{for all }i<k;
\label{eqsdevbmhbfHr}
\end{equation}
whereby $\pair w z=\pair{t_0}{t_k}\in\Theta$ by  the  transitivity of $\Theta$. Hence, $\phi(x)\leq \Theta$. Thus,
\begin{equation}
\text{for $x=\psi(\Theta)$ defined in \eqref{eqdizTrNvq}, $\,\,\phi(x) = \Theta$.}
\label{eqtxtzBmTnSx}
\end{equation}
This proves the surjectivity of $\phi$. 

It is clear from \eqref{eqwdfPhhns} that $\phi$ is order-preserving, that is, $x\leq y\in D$ implies that $\phi(x)\leq \phi(y)$. To show that $\phi$ is injective, assume that $x\neq y\in D$. Since $x=\bigvee(J(D)\cap \ideal x)$ and similarly for $y$, 
there exists an $a\in J(D)$ such that  $a\leq x$ and 
$a\nleq y$, or conversely. So, we can assume that  $a\leq x$ and $a\nleq y$. Since $\szn\colon \Prime(C)\to J(D)$ is a surjective map by Definition~\ref{defsznzSj}, there exists an edge $\inp\in\Prime(C)$ such that $\szn(\inp)=a$. Then $\pair{0_\inp}{1_\inp}$ is in $\phi(x)$ but not in $\phi(y)$. This proves the injectivity of $\phi$.

Now that we know that $\phi$ is a bijection, its clear by \eqref{eqdizTrNvq} and \eqref{eqtxtzBmTnSx} that the map 
\begin{equation}
\text{$\psi\colon\Con(\alga)\to D$, defined by 
$\Theta\mapsto x$ in 
 \eqref{eqdizTrNvq}, is the inverse of $\phi$.}
\label{eqtxtzbpstnvR}
\end{equation}
It is obvious by  \eqref{eqdizTrNvq} that $\psi$ is order-preserving. Consequently, $\phi$ is an order isomorphism and, thus, a lattice isomorphism.

We claim that,
\begin{equation}
\text{for every $\inp\in \Prime(C)$,\quad $\szn(\inp)=\psi(\con(0_\inp,1_\inp))$.}
\label{eqtxtzBkQkK}
\end{equation}
Since $\con(0_\inp,1_\inp)$ obviously collapses $\pair{0_\inp}{1_\inp}$, we obtain from 
 \eqref{eqdizTrNvq}  and \eqref{eqtxtzbpstnvR} that  
$\szn(\inp)\leq \psi(\con(0_\inp,1_\inp))$. Conversely, with the notation $a:=\szn(\inp)$, it is clear by \eqref{eqwdfPhhns} that $\phi(a)\ni \pair{0_\inp}{1_\inp}$.  Hence, $\phi(a)\geq \con(0_\inp,1_\inp)$. Using that $\psi$ is order-preserving, we obtain that $\szn(\inp)=a=\psi(\phi(a))\geq \psi(\con(0_\inp,1_\inp))$, proving \eqref{eqtxtzBkQkK}.

Next, let $w<z$ in $C$, and assume that $\pair w z\in\Theta=\phi(x)$ for some $x\in D$. Since \eqref{eqsdevbmhbfHr} holds for the elements $t_i$, see \eqref{eqtxthtBfk}, $\con(t_i,t_{i+1})\leq \Theta$ for all $i<k$. This also holds for $\Theta:=\con(w,z)$, whereby we obtain easily that 
\begin{equation}
\text{if $w<z\in C$,\quad then\quad$\con(w,z)=\bigvee_{i<k} \con(t_i,t_{i+1})$.}
\label{eqtxtzTrhSd}
\end{equation}

Now, we are in the position to show that $\phi(Q)=\Princ(\alga)$. Assume that $x\in Q$. Since 
 $Q=\srep(C,\szn, D)$, there are $w<z$ such that for the elements defined in \eqref{eqsdevbmhbfHr}, 
\begin{equation}
x =  \erep([w,z])  \overset{\eqref{eqsdhsdrepRtp}}=  \bigvee_{i<k}\szn([t_i,t_{i+1}]).
\label{eqdizTnPxCv}
\end{equation}
It follows from \eqref{eqtxtzBkQkK} that $\psi$, which is  a lattice isomorphism, maps the right-hand side of \eqref{eqtxtzTrhSd} to that of \eqref{eqdizTnPxCv}. Hence, $\psi(\con(w,z))=x$, which gives that $\phi(x)=\con(w,z)\in\Princ(\alga)$.
Consequently, $\phi(Q)\subseteq \Princ(\alga)$. 

Finally, we need to exclude that  this inclusion is proper. Suppose to the contrary that $\phi(Q)\subset \Princ(\alga)$, and pick a principal congruence from $\Princ(\alga)\setminus\phi(Q)$. Since $\phi\colon D\to\Con(\alga)$ is surjective, this congruence is of the form $\phi(x)$ where $x\notin Q$. Since $\phi(x)$ is a principal congruence,  there are $w<z\in C$ such that  $\phi(x)$ is of the form $\con(w,z)$, described in \eqref{eqtxtzTrhSd}. 
Taking the $\psi$-images of both sides of the equation in \eqref{eqtxtzTrhSd} and using \eqref{eqtxtzBkQkK} in the same way as before, we obtain the validity of \eqref{eqdizTnPxCv}.  Hence, $x=\erep([w,z])\in  \srep(C,\szn,D) =Q$, contradicting the choice of $x$. This proves the equality $\phi(Q)= \Princ(\alga)$ and completes the proof of Proposition~\ref{propositionbchrsnlsgR}.
\end{proof}

\color{black}

%
%
%
%
%
%
%


\begin{thebibliography}{99}

\bibitem{czgonemap}
G. Cz\'edli: 
 Representing a monotone map by principal lattice congruences. 
Acta Mathematica Hungarica \tbf{147}  (2015), 12--18

\bibitem{czgaleph0}
G. Cz\'edli: 
The ordered set of principal congruences of a countable lattice, Algebra Universalis 75 (2016), 351--380

\bibitem{czgprincout}
G. Cz\'edli: 
 An independence theorem for ordered sets of principal congruences and automorphism groups of bounded lattices, Acta Sci. Math. (Szeged) 82 (2016), 3--18

\bibitem{czgmanymaps}
G. Cz\'edli: 
Representing some families of monotone maps by principal lattice congruences. Algebra Universalis 77 (2017), 51--77

\bibitem{czgdiagr}
G. Cz\'edli: Diagrams and rectangular extensions of planar semimodular lattices. Algebra Universalis, online from March 31, 2017, DOI 10.1007/s00012-017-0437-0


\bibitem{czgcometic}
G. Cz\'edli: 
 Cometic functors and representing order-preserving maps by principal lattice congruences. Algebra Universalis, submitted


\bibitem{czgDQprincrep}
G. Cz\'edli: 
Characterizing fully principal congruence representable distributive lattices. \texttt{http://arxiv.org/abs/1706.03401}

\bibitem{czgggchapter}
Cz\'edli, G.; Gr\"atzer, G.: 
Planar semimodular lattices
and their diagrams. Chapter 3 in:
Gr\"atzer, G., Wehrung, F. (eds.) Lattice Theory: Special Topics and Applications. 
Birkh\"auser Verlag, Basel (2014)


\bibitem{czgschtvisualI}
Cz\'edli, G. and Schmidt, E. T.:
Slim semimodular lattices. I. A visual approach, Order \tbf{29}, 481--497 (2012)

\bibitem{ggGLT}
Gr\"atzer, G.: Lattice Theory: Foundation. Birkh\"auser Verlag, Basel (2011)


\bibitem{gG13}
Gr\"atzer, G.:
The order of principal congruences of a bounded lattice. Algebra Universalis 70,
95--105 (2013)


\bibitem{ggCFL2}
Gr\"atzer, G.:
The Congruences of a Finite Lattice, 
A \emph{Proof-by-Picture} Approach,
second edition.
Birkh\"auser, 2016 


\bibitem{gGasm2016}
Gr\"atzer, G..:
Homomorphisms and principal congruences of bounded lattices I. Isotone maps of principal congruences. Acta Sci. Math. 82 (2016) 353--360  


\bibitem{ggwith1}
G. Gr\"atzer: 
Characterizing representability by principal congruences for finite distributive lattices with a join-irreducible unit element.
Acta Sci. Math. (Szeged), \dots . Available from  ResearchGate

\bibitem{gGprincII}
Gr\"atzer, G.:
 Homomorphisms and principal congruences of bounded lattices. II. Sketching the proof for sublattices. Available from  ResearchGate


\bibitem{gGprincIII}
Gr\"atzer, G.:
Homomorphisms and principal congruences of bounded lattices. III. The Independence Theorem. Available from  ResearchGate


\bibitem{gGknappI}
Gr\"atzer, G. and Knapp, E.: 
Notes on planar semimodular lattices. I. Construction. Acta Sci. Math. (Szeged)
\tbf{73}, 445--462 (2007)


\bibitem{gratzerlakser}
Gr\"atzer, G. and Lakser, H.:
Some preliminary results on
the set of principal congruences of a finite lattice. Algebra Universalis, to appear.

\bibitem{kellyrival}
Kelly, D., Rival, I.: 
Planar lattices. 
Canad. J. Math. \tbf{27}, 636--665 (1975)


\end{thebibliography}
\end{document}